\documentclass[11pt]{article}
\usepackage{amsthm,fullpage,graphics, xcolor, url, hyperref}
\usepackage{amssymb, amsmath}
\usepackage{graphicx}
\usepackage{comment}
\usepackage{tikz}
\usepackage[font=footnotesize]{caption}
\usetikzlibrary{calc}

\tikzstyle{every node}=[circle, draw, fill=black,inner sep=0pt, minimum width=4pt]

\newtheorem{lem}{Lemma}
\newtheorem{cor}[lem]{Corollary}
\newtheorem{thm}[lem]{Theorem}

\newtheorem{conj}{Conjecture}
\theoremstyle{definition}

\newtheorem{defn}{Definition}
\def\aftermath{\par\vspace{-\belowdisplayskip}\vspace{-\parskip}\vspace{-\baselineskip}}

\RequirePackage{marginnote,hyperref}
\newcommand{\aaside}[2]{\marginnote{\scriptsize{#1}}[#2]}

\newcommand\EmphE[2]{\emph{#1}\aaside{#1}{#2}}

\begin{document}
\def\F{\mathcal{F}}
\def\Z{\mathbb{Z}}
\def\ex{\textrm{ex}}
\def\exP{\textrm{ex}_{\mathcal{P}}}
\def\Turan{Tur\'{a}n}

\title{Planar {\Turan} Numbers of Cycles:\\ A Counterexample}

\date{}
\author{Daniel W. Cranston\thanks{%
Department of Computer Science, Virginia Commonwealth
University, Richmond, VA, USA;
\texttt{dcranston@vcu.edu}
} 
\and 
Bernard Lidick\'{y}\thanks{Iowa State University, Department of 
Mathematics, Iowa State University, Ames, IA, USA; \newline
\texttt{lidicky@iastate.edu}. Research of this author is partially supported by NSF grant DMS-1855653.}
\and Xiaonan Liu\thanks{School of Mathematics, Georgia Institute of Technology, Atlanta, GA, USA; \texttt{xliu729@gatech.edu}} 
\and Abhinav Shantanam \thanks{Department of Mathematics, 
Simon Fraser University, Burnaby, BC, Canada;
\texttt{ashantan@sfu.ca}}
}

\maketitle
\begin{abstract}
The planar {\Turan} number $\textrm{ex}_{\mathcal{P}}(C_{\ell},n)$ is the largest
number of edges in an $n$-vertex planar graph with no $\ell$-cycle.  For
$\ell\in \{3,4,5,6\}$, upper bounds on $\textrm{ex}_{\mathcal{P}}(C_{\ell},n)$ 
are known that hold with equality infinitely
often.  Ghosh, Gy\"{o}ri, Martin, Paulo, and Xiao [arxiv:2004.14094] conjectured an upper bound on
$\textrm{ex}_{\mathcal{P}}(C_{\ell},n)$ for every $\ell\ge 7$ and $n$ sufficiently
large.  We disprove this conjecture for every $\ell\ge 11$.  We also propose two revised versions of the conjecture.
\end{abstract}

\section{Introduction}
The {\Turan} number $\ex(n,H)$ for a graph $H$ is the maximum number of edges in
an $n$-vertex graph with no copy of $H$ as a subgraph.  {\Turan} famously showed
that $\ex(n,K_{\ell})\le (1-\frac1{\ell-1})\frac{n^2}2$; for example, see
\cite[Chapter 32]{PFTB}.  The Erd\H{o}s--Stone
Theorem~\cite[Exercise 10.38]{lovasz-problems-book}
generalizes this result, by asymptotically determining $\ex(n,H)$ for every
non-bipartite graph $H$: $\ex(n,H)=(1-\frac1{\chi(H)-1})\frac{n^2}2+o(n^2)$; here
$\chi(H)$ is the chromatic number of $H$.  Dowden~\cite{dowden} considered the
problem when restricting to $n$-vertex graphs that are planar.
The \emph{planar {\Turan} number} $\exP(n,H)$\aaside{$\exP(n,H)$}{-4mm} for a graph $H$ is the maximum number of
edges in an $n$-vertex planar graph with no copy of $H$ as a subgraph (not
necessarily induced).   This parameter has been investigated for various graphs $H$ in~\cite{LSS2} and~\cite{FZW}; but here we focus mainly on cycles.
It is well-known that if $G$ is an $n$-vertex planar
graph with no triangle, then $G$ has at most $2n-4$ edges; further, this bound
is achieved by every planar graph with each face of length 4.  Thus,
$\exP(n,C_3)=2n-4$ for all $n\ge 4$.  Dowden~\cite{dowden}
proved that $\exP(n,C_4)\le\frac{15(n-2)}7$ for all $n\ge 4$ and
$\exP(n,C_5)\le\frac{12n-33}5$ for all $n\ge 11$.  He also gave constructions
showing that both of these bounds are sharp infinitely often.

For each $k\in \{4,5\}$, form $\Theta_k$ from $C_k$ by adding a chord of the cycle.  
Lan, Shi, and Song~\cite{LSS} showed 
that $\exP(n,\Theta_4)\le \frac{12(n-2)}5$ for all $n\ge 4$, 
that $\exP(n,\Theta_5)\le \frac{5(n-2)}2$ for all $n\ge 5$, and
that $\exP(n,C_6)\le \frac{18(n-2)}7$ for all $n\ge 7$. 
The bounds for $\Theta_4$ and $\Theta_5$ are sharp infinitely often.
However, the bound for $C_6$ was
strengthened by 
Ghosh, Gy\"{o}ri, Martin, Paulos, and Xiao~\cite{GGMPX}, who showed
that $\exP(n,C_6)\le \frac{5n-14}2$ for all $n\ge 18$.  
They also showed that this bound is sharp infinitely often.
In the same paper, Ghosh
et al. conjectured a bound on $\exP(n,C_{\ell})$ for each $\ell\ge 7$ and each
sufficiently large $n$.  In this note, we disprove their conjecture.

\begin{conj}[\cite{GGMPX}; now disproved]
\label{main-conj}
For each $\ell\ge 7$, for $n$ sufficiently large, if $G$ is an $n$-vertex planar
graph with no copy of $C_{\ell}$, then $e(G)\le
\frac{3(\ell-1)}{\ell}n-\frac{6(\ell+1)}{\ell}$.  
That is, $\exP(n,C_{\ell})\le \frac{3(\ell-1)}{\ell}n-\frac{6(\ell+1)}{\ell}$.  
\end{conj}

In fact, we disprove the conjecture in a strong way.

\begin{thm}
\label{over-thm}
For each $\ell\ge 11$ and each $n$ sufficiently large
(as a function of $\ell$), we have
$\exP(n,C_{\ell}) > \frac{3(\ell-1)}{\ell}n-\frac{6(\ell+1)}{\ell}$.  
Furthermore, if there exists a function $s:\Z^+\to \Z^+$ such that
$\exP(n,C_{\ell})\le \frac{3(s(\ell)-1)}{s(\ell)}n$ for all $\ell$ and all $n$
sufficiently large (as a function of $\ell$), then
$s(\ell)=\Omega(\ell^{\lg_23})$.
\end{thm}

We prove the first statement of Theorem~\ref{over-thm} in
Section~\ref{first-sec}, and sketch a proof of the second statement in
Section~\ref{second-sec}.  Our constructions modify that outlined by Ghosh et
al.~\cite{GGMPX}. The main building blocks, which we call \emph{gadgets},
are triangulations, in which every cycle has length less than $\ell$. 
Clearly, a set of vertex-disjoint gadgets will have no $C_{\ell}$.
To increase the average degree,
we can identify vertices on the outer faces of these gadgets as long as
we avoid creating cycles.  We can also allow ourselves to create cycles
among the gadgets as long as each created cycle has length more than $\ell$.
So we must find the way to do this most efficiently.  

Our notation is standard, but for completeness we record a few things
here.  We let $e(G)$ and $n(G)$ denote the numbers of edges and vertices in a
graph $G$.  We write $C_{\ell}$ for a cycle \mbox{of length $\ell$.}

\section{Disproving the Conjecture: a First Construction}
\label{first-sec}

To disprove Conjecture~\ref{main-conj}, we start with
a planar graph in which each face has length $\ell+1$ (and each cycle has
length at least $\ell+1$), and then we ``substitute'' a gadget for each vertex.  
As a first step, we construct
the densest planar graphs with a given girth $g$, for
each fixed $g\ge 6$.
We will also need that our dense graphs have maximum degree 3, as we require
in the following definition.

\begin{defn}
If $G$ is a planar graph of girth $g$ with each vertex of degree 2 or 3, 
and $e(G)=\frac{g}{g-2}(n-2)$, then $G$ is a \emph{dense graph of girth
$g$}\aaside{dense graph}{-5mm}.
\end{defn}

An easy counting argument shows that if $G$ is an $n$-vertex dense graph of
girth $g$, where $n=(5g-10)\frac{k}{2}-g+4$ (for some positive even integer
$k$), then $G$ has $10k-8$ vertices of degree 3 and all other vertices of
degree 2.

\begin{lem}
Fix an integer $g\ge 3$.  
If $G$ is an $n$-vertex planar graph with girth $g$, then $e(G)\le
\frac{g}{g-2}(n-2)$.  For each $g\ge 6$, this bound holds with equality
infinitely often; specifically, it holds with equality if $k$ is a positive even
integer and $n=(5g-10)\frac{k}{2}-g+4$.  In fact, for each such $k$ and $n$,
there exists a graph $G$ that attains this bound and that has every vertex of degree 2
or 3.
\label{dense-lem}
\end{lem}

\begin{proof}[Proof of Lemma~\ref{dense-lem}.]
Let $G$ be an $n$-vertex planar graph with girth $g$.  
Denote by $n$, $e$, and $f$ the numbers of vertices, edges, and faces in $G$.
Every face boundary contains a cycle,\footnote{To see this, form $G'$ from $G$
by deleting all cut-edges.  Since each component of $G'$ is 2-connected,
each face boundary \emph{is} either a cycle or a disjoint union of cycles (if $G'$ is disconnected).  Note that each face boundary in $G$
contains all edges of a face boundary in $G'$.}
so every face boundary has length at least $g$.  Thus, $2e\ge gf$.  Substituting into
Euler's formula and simplifying gives the desired bound: $e\le
\frac{g}{g-2}(n-2)$.

Now we construct graphs for which the bound holds with equality.
Before giving our full construction, we sketch a simpler construction which has the desired properties except that it has maximum degree 6 (rather than each degree
being 2 or 3, as we require).  Begin with a 4-connected $n$-vertex planar
triangulation with maximum degree 6.  We will find a set
$M$ of edges such that every triangular face contains exactly one edge in $M$.
To see that such a set exists, we consider the planar dual $G^*$.  Since $G$ is a triangulation and 2-connected, $G^*$ is $3$-regular.
By Tutte's Theorem, $G^*$ contains a perfect matching $M^*$ (in fact, this was
proved earlier by Petersen).  The set $M$ of edges in $G$ corresponding to the edges
of $M^*$ in $G^*$ has the desired property: each triangle of $G$ contains exactly one
edge of $M$.

To get the desired graph $G'$ with each face of length $g$, 
we replace each edge of $G$ not in $M$ with a path of
length $\lfloor (g+1)/3\rfloor$ and replace each edge of $G$ in $M$ with a path of
length $g-2\lfloor (g+1)/3\rfloor$.  Now each face of $G'$ has length $2\lfloor
(g+1)/3\rfloor+(g-2\lfloor (g+1)/3\rfloor)=g$.  Thus, for $G'$ the inequality
$2e(G')\ge gf(G')$ in the initial paragraph holds with equality.  So
$e(G')=\frac{g}{g-2}(n(G')-2)$.
Since each non-facial cycle of $G$ has length at least 4, 
each non-facial cycle of $G'$ has length at least $g$.

Now we show how to also guarantee that each vertex of $G'$ has degree 2 or 3.
The construction is similar, except that it starts from a particular planar graph $G$ with
every face of length 6 and every vertex of degree 2 or 3.  Again, we find a
subset $M$ of edges such that each face of $G$ contains exactly one edge of $M$.
To form $G'$ from $G$, we replace each edge not in $M$ with a path of length
$\lfloor(g+1)/6\rfloor$ and we replace each edge in $M$ with a path of length
$g-5\lfloor(g+1)/6\rfloor$.  Thus, each face of $G'$ has length exactly
$5\lfloor(g+1)/6\rfloor + (g-5\lfloor(g+1)/6\rfloor)=g$.  

It will turn out that each non-facial cycle of $G$ has either (i) length at
least 10 or (ii) length at least 8 and at least one edge in $M$. 
The corresponding non-facial cycle in $G'$ thus has length at least $g$. 
In Case (ii) this follows from the calculation in the previous paragraph. 
In Case (i), when $g \ge 10$ this holds because $10 \lfloor(g+1)/6\rfloor \ge
10 (g-4)/6 \ge g$.
So consider Case (i) when $g\le 9$.
Since each path in $G'$ replacing an edge in $G$ has length at least 1, each
non-facial cycle in $G'$ has length at least 10, which is at least $g$ since $g\le 9$.  
Thus, what remains is to construct our graph $G$, specify the set of
edges $M$, and check that each non-facial cycle in $G$ either has length at least 10 or has length 8 and includes an edge in $M$.

\def\r{1.0}
\def\sx{0.0}
\def\sy{0.0}
\def\nx{6}
\def\ny{4}
\def\hnxl{2.5}
\def\hnxh{3.5}
\def\matchingwidth{1.0mm}
\def\matchingcol{blue!70!white}

\begin{figure}[!t]
\centering

\begin{tikzpicture}[baseline, scale=0.50,
nodelabel/.style={fill=none,draw=none}, very thick]

\foreach \x in {1,2,...,\nx} 
{
\foreach \y in {1,2,...,\ny} 
{
\if\x3
\else
\draw (\sx+\y*\r*0.866+\x*2*\r*0.866+\r*0.866,\sy-\y*1.5*\r+\r*0.5) -- (\sx+\y*\r*0.866+\x*2*\r*0.866+\r*0.0,\sy-\y*1.5*\r+\r*1.0);
\draw (\sx+\y*\r*0.866+\x*2*\r*0.866+\r*0.0,\sy-\y*1.5*\r+\r*1.0) -- (\sx+\y*\r*0.866+\x*2*\r*0.866-\r*0.866,\sy-\y*1.5*\r+\r*0.5);
\draw (\sx+\y*\r*0.866+\x*2*\r*0.866-\r*0.866,\sy-\y*1.5*\r+\r*0.5) -- (\sx+\y*\r*0.866+\x*2*\r*0.866-\r*0.866,\sy-\y*1.5*\r-\r*0.5);
\draw (\sx+\y*\r*0.866+\x*2*\r*0.866-\r*0.866,\sy-\y*1.5*\r-\r*0.5) -- (\sx+\y*\r*0.866+\x*2*\r*0.866+\r*0.0,\sy-\y*1.5*\r-\r*1.0);
\draw (\sx+\y*\r*0.866+\x*2*\r*0.866+\r*0.0,\sy-\y*1.5*\r-\r*1.0) -- (\sx+\y*\r*0.866+\x*2*\r*0.866+\r*0.866,\sy-\y*1.5*\r-\r*0.5);
\draw (\sx+\y*\r*0.866+\x*2*\r*0.866+\r*0.866,\sy-\y*1.5*\r-\r*0.5) -- (\sx+\y*\r*0.866+\x*2*\r*0.866+\r*0.866,\sy-\y*1.5*\r+\r*0.5);
\fi
}
}

\foreach \x in {1,2,...,\hnxl}
{
\draw (\sx-\x*2*\r*0.866,\sy-\r*0.5) to [out=90,in=90] (\sx+\x*2*\r*0.866+\r*0.866,\sy-\r*0.5);
\draw (\sx-\x*2*\r*0.866,\sy-\r*0.5) -- (\sx-\x*2*\r*0.866,\sy-\ny*\r*1.5-\r);
\draw (\sx-\x*2*\r*0.866,\sy-\ny*\r*1.5-\r) to [out=270,in=270] (\sx+\ny*\r*0.866+\x*2*\r*0.866,\sy-\ny*\r*1.5-\r);
}

\foreach \x in {1,2,...,\hnxh}
{
\draw (\sx+\nx*2*\r*0.866+\ny*\r*0.866-\x*2*\r*0.866-\r*0.866,\sy-\r*0.5) to [out=90,in=90] (\sx+\nx*2*\r*0.866+\ny*\r*0.866+\x*2*\r*0.866+2.5*\r*0.866,\sy-\r*0.5);
\draw (\sx+\nx*2*\r*0.866+\ny*\r*0.866+\x*2*\r*0.866+2.5*\r*0.866,\sy-\r*0.5) -- (\sx+\nx*2*\r*0.866+\ny*\r*0.866+\x*2*\r*0.866+2.5*\r*0.866,\sy-\ny*\r*1.5-\r);
\draw (\sx+\nx*2*\r*0.866+\ny*\r*0.866-\x*2*\r*0.866+2*\r*0.866,\sy-\ny*\r*1.5-\r) to [out=270,in=270] (\sx+\nx*2*\r*0.866+\ny*\r*0.866+\x*2*\r*0.866+2.5*\r*0.866,\sy-\ny*\r*1.5-\r);
}

\draw (\sx+3.2*\r*0.866,\sy-.475*\r)node[left, inner sep=8pt,
nodelabel]{\footnotesize{$v_{1, 2}$}};
\draw (\sx+6.15*\r*0.866,\sy-7.15*\r)node[left, inner sep=4pt,
nodelabel]{\footnotesize{$v_{1, 10}$}};
\draw (\sx+\nx*2*\r*0.866+\ny*\r*0.866-3*\r*0.866,\sy-\r*0.5)node[right, inner
sep=4pt, nodelabel]{\footnotesize{$v_{k, 1}$}};
\draw (\sx+\nx*2*\r*0.866+\ny*\r*0.866+\r*0.866-1.05,\sy-7.15*\r)node[right,
inner sep=8pt, nodelabel]{\footnotesize{$v_{k, 9}$}};

\foreach \x in {1,2,...,\nx}
{
\foreach \y in {1,3,...,\ny}
{
\if\x3
\else
\draw[\matchingcol,line width=\matchingwidth] (\sx+\y*\r*0.866+\x*2*\r*0.866+\r*0.0,\sy-\y*1.5*\r-\r*1.0) -- (\sx+\y*\r*0.866+\x*2*\r*0.866+\r*0.866,\sy-\y*1.5*\r-\r*0.5);
\fi
}
}

\draw[\matchingcol,line width=\matchingwidth] (\sx+2*\r*0.866,\sy-\r)
to [out=200,in=180,looseness=1.8] (\sx+\ny*\r*0.866,\sy-\ny*\r*1.5+\r);
\foreach \x in {2,...,2}
{
\draw[\matchingcol,line width=\matchingwidth] (\sx-\x*2*\r*0.866,\sy-\r*0.5) to [out=90,in=90] (\sx+\x*2*\r*0.866+\r*0.866,\sy-\r*0.5);
\draw[\matchingcol,line width=\matchingwidth] (\sx-\x*2*\r*0.866,\sy-\r*0.5) -- (\sx-\x*2*\r*0.866,\sy-\ny*\r*1.5-\r);
\draw[\matchingcol,line width=\matchingwidth] (\sx-\x*2*\r*0.866,\sy-\ny*\r*1.5-\r) to [out=270,in=270] (\sx+\ny*\r*0.866+\x*2*\r*0.866,\sy-\ny*\r*1.5-\r);
}

\draw[\matchingcol,line width=\matchingwidth]
(\sx+\nx*2*\r*0.866+\ny*\r*0.866-\r*0.866,\sy-\r*2.5) to [out=0,in=20,
looseness=2.1] (\sx+\nx*2*\r*0.866+\ny*\r*0.866+\r*0.866,\sy-\ny*\r*1.5-\r*0.5);
\foreach \x in {2,...,2}
{
\draw[\matchingcol,line width=\matchingwidth] (\sx+\nx*2*\r*0.866+\ny*\r*0.866-\x*2*\r*0.866-\r*0.866,\sy-\r*0.5) to [out=90,in=90] (\sx+\nx*2*\r*0.866+\ny*\r*0.866+\x*2*\r*0.866+2.5*\r*0.866,\sy-\r*0.5);
\draw[\matchingcol,line width=\matchingwidth] (\sx+\nx*2*\r*0.866+\ny*\r*0.866+\x*2*\r*0.866+2.5*\r*0.866,\sy-\r*0.5) -- (\sx+\nx*2*\r*0.866+\ny*\r*0.866+\x*2*\r*0.866+2.5*\r*0.866,\sy-\ny*\r*1.5-\r);
\draw[\matchingcol,line width=\matchingwidth] (\sx+\nx*2*\r*0.866+\ny*\r*0.866-\x*2*\r*0.866+2*\r*0.866,\sy-\ny*\r*1.5-\r) to [out=270,in=270] (\sx+\nx*2*\r*0.866+\ny*\r*0.866+\x*2*\r*0.866+2.5*\r*0.866,\sy-\ny*\r*1.5-\r);
}

\foreach \x in {1,2,...,\nx} 
{
\foreach \y in {1,2,...,\ny} 
{
\if\x3
\draw (\sx+\y*\r*0.866+\x*2.03*\r*0.866+\r*0.0,\sy-\y*1.5*\r)node[inner sep=0pt, nodelabel]{\ldots};
\else
\draw (\sx+\y*\r*0.866+\x*2*\r*0.866+\r*0.866,\sy-\y*1.5*\r+\r*0.5)node{};
\draw (\sx+\y*\r*0.866+\x*2*\r*0.866+\r*0.0,\sy-\y*1.5*\r+\r*1.0)node{};
\draw (\sx+\y*\r*0.866+\x*2*\r*0.866-\r*0.866,\sy-\y*1.5*\r+\r*0.5)node{};
\draw (\sx+\y*\r*0.866+\x*2*\r*0.866-\r*0.866,\sy-\y*1.5*\r-\r*0.5)node{};
\draw (\sx+\y*\r*0.866+\x*2*\r*0.866+\r*0.0,\sy-\y*1.5*\r-\r*1.0)node{};
\draw (\sx+\y*\r*0.866+\x*2*\r*0.866+\r*0.866,\sy-\y*1.5*\r-\r*0.5)node{};
\if\x4
\else
\ifnum\x<3
\pgfmathtruncatemacro\k{\x*5-5+\y+1}
\draw (\sx+\y*\r*0.866+\x*2*\r*0.866+\r*0.0,\sy-\y*1.5*\r)node[inner sep=0pt,
nodelabel]{\footnotesize{$C_{\k}$}};
    \if\x1
    \if\y4
    \draw (\sx+5*\r*0.866+\x*2*\r*0.866+\r*0.0,\sy-5*1.5*\r)node[inner sep=0pt,
nodelabel]{\footnotesize{$C_6$}};
    \fi
    \fi
\fi
\fi
\fi
}
}

\draw (\sx+1*\r*0.866+1*2*\r*0.866-\r*0.866,\sy-1*1.5*\r+\r*0.5) node[minimum size = 6.5pt]{};
\draw (\sx+5*\r*0.866+1*2*\r*0.866-\r*0.866,\sy-5*1.5*\r+\r*0.5) node[minimum size = 6.5pt]{};
\draw (\sx+1*\r*0.866+6*2*\r*0.866+\r*0.0,\sy-1*1.5*\r+\r*1.0)   node[minimum size = 6.5pt]{};
\draw (\sx+5*\r*0.866+6*2*\r*0.866-\r*0.0,\sy-5*1.5*\r+\r*1.0)   node[minimum size = 6.5pt]{};

\draw (\sx+4*\r*0.866+1*2*\r*0.866+\r*0.0-4*\r*0.866,\sy-4*1.5*\r)node[inner
sep=0pt, nodelabel]{\footnotesize{$C_0$}};
\draw (\sx+4*\r*0.866+1*2*\r*0.866+\r*0.0-4*\r*0.866,\sy-2.2*1.5*\r)node[inner
sep=0pt, nodelabel]{\footnotesize{$C_1$}};
\draw (\sx+1*\r*0.866+6*2*\r*0.866+\r*0.0+4.5*\r*0.866,\sy-1*1.5*\r)node[inner
sep=0pt, nodelabel]{\footnotesize{$C_{5k-3}$}};
\draw (\sx+1*\r*0.866+6*2*\r*0.866+\r*0.0+4.5*\r*0.866,\sy-3*1.5*\r)node[inner
sep=0pt, nodelabel]{\footnotesize{$C_{5k-4}$}};

\end{tikzpicture}

\caption{The planar graph $G_k$ has $10k-2$ vertices, $15k-6$ edges, and every
face of length 6. Every vertex of $G_k$ has degree 2 or 3 and every non-facial
cycle either (i) has length at least 10 or (ii) has length 8 and includes a blue
edge. The set of blue edges intersects every face exactly once.}

\end{figure}
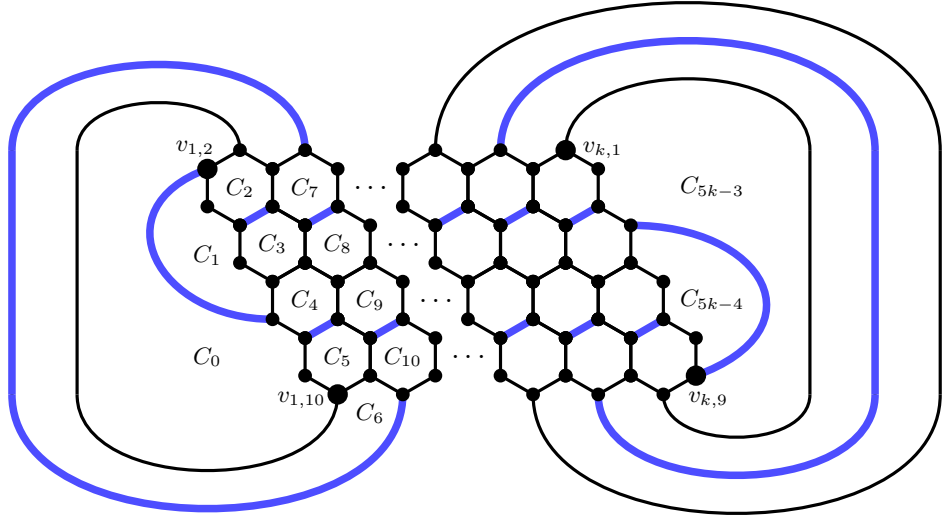

We construct an infinite family of planar graphs $G_k$ on $10k-2$ vertices, with
$5k-2$ faces (each of length 6), and with all vertices of degree 2 or 3; here
$k$ is an arbitrary positive even integer.  Figure~1 shows $G_k$.  (By
Euler's formula, each $G_k$ has 6 vertices of degree 2 and $10k-8$
vertices of degree 3.)  Each of
$k$ ``diagonal columns'' contains 10 vertices, except for the first and last,
which each contain one vertex fewer.  We write $v_{i,j}$ to denote the $j$th
vertex down from the top in column $i$, except that we start column 1 with
$v_{1,2}$.  So $V=\{v_{i,j}~|~1\le i\le k,~1\le j\le 10,
(i,j)\notin\{(1,1),(k,10)\}\}$.
The edge set consists of the boundary
cycles of $4(k-1)$ 6-faces in the hexagonal grid, $k-1$ ``curved edges''
$v_{i,1}v_{i-1,10}$, when $2\le i\le k$, as well two ``end edges'' $v_{1,2}v_{1,7}$ and
$v_{k,4}v_{k,9}$.  The matching $M$ contains $v_{i,4}v_{i+1,3}$ and
$v_{i,8}v_{i+1,7}$ when $1\le i\le k-1$, edge $v_{i,1}v_{i-1,10}$ for
each odd $i\ge 3$ if $k\ge 4$, and the end edges $v_{1,2}v_{1,7}$ and $v_{k,4}v_{k,9}$. 
It is easy to check that the only vertices with degree 2 are $v_{1,3}, v_{1,5},
v_{1,9}, v_{k,2}, v_{k,6}, v_{k,8}$; the remaining $10k-8$ vertices all
have degree 3.

We now show that every non-facial cycle has either (i) length at least 10 or
(ii) length at least 8 and at least one edge in $M$. 
We denote by $C_2, C_3, \ldots, C_{5k-5}$ the facial cycles that do not use any
end-edge.  
Informally, $C_2$ is the ``top left'' of these (containing $v_{1, 2}$), and
subscripts increase as we move down the first diagonal and 
then wrap around toroidally with the facial cycle containing $v_{1, 10}$ and 
two curved edges (see Figure~1), and continue on to the facial cycle containing $v_{k, 9}$. 
Formally, each of these is $C_k$, where $X$ denotes its vertex
set and $k := \max\{j/2: v_{i, j} \in X\} + 5 * \min\{i-1: v_{i, j} \in X\} +
(|\{i: v_{i, j} \in X\}|-2)$. 
The facial cycles containing the left end-edge are $C_0$ and $C_1$, and 
those containing the right end-edge are $C_{5k-4}$ and $C_{5k-3}$.

Note that the edge-set of any non-facial cycle $C$
is the symmetric difference of the edge-sets of the facial cycles 
``inside'' (or ``outside'') of $C$.  Consider first a non-facial cycle $C$
that does not contain any end-edge.  Pick the side of $C$ that
does not contain the right end-edge; take the symmetric difference of the
edge-sets of the facial cycles on this 
side incrementally, in order of increasing subscripts. The symmetric difference 
of the first two facial cycles has size at least 10 and this size never
decreases.  Now consider the non-facial 
cycles that contain exactly one end-edge; by (rotational) symmetry, assume it is
the left end-edge. For these cycles, take 
the symmetric difference incrementally as above for the side not containing the 
right end-edge; the symmetric difference of the first two facial cycles has
size at least 8 and again this size never decreases. 

Finally, consider a non-facial cycle $C$ that 
contains both end-edges.  Now take the symmetric difference incrementally 
as above for the side of $C$ that includes $C_1$; the size of the symmetric difference
is now initially at least 8, and never decreases until the final facial cycle ($C_{5k-4}$ or 
$C_{5k-3}$) is added and the symmetric difference is complete. The final
facial cycle $C'$ may reduce the size of the symmetric difference by at most 4, 
but the final symmetric difference still has size at least 12
(due to the position of $C'$ relative to $C_1$, and the fact that $k \geq 2$). 

To finish the proof, we should verify that $|V(G')|=(5g-10)\frac{k}{2}-g+4$, as claimed.
By construction, each vertex of $G'$ has degree 2 or 3.  Each vertex with
degree 3 in $G'$ also has degree 3 in $G$, and we have exactly $10k-8$ of
these.  Let $n$, $e$, and $f$ denote the numbers of vertices, edges, and
faces in $G'$.  Now summing degrees gives
$$
    3(10k-8)+ 2 (n-(10k-8))=2e=gf=\frac{g}{g-2} (2n-4),
$$
where the last two equalities hold as 
at the start of the proof.
Thus, $n=(5g-10)\frac{k}{2}-g+4$.
\end{proof}

\begin{defn}
\label{sub-defn}
Let $G$ be a 2-connected plane graph, with every vertex of degree 2 or 3.
Let $B$ be a plane graph with 3 vertices specified on its outer face.
To \EmphE{substitute $B$ into $G$}{-4mm} we do the following.  Subdivide every edge of
$G$.  For each vertex $v$ in $G$, delete $v$ from the subdivided graph and
identify $d(v)$ vertices on the outer face of a copy of $B$ with the neighbors
of $v$ in the subdivided graph.
\end{defn}

Now we consider the result of substituting $B$ into $G$, as in
Definition~\ref{sub-defn}.
\begin{lem}
\label{lem:construction}
Let $G$ be a plane graph; denote by $n_2$ and $n_3$ the numbers of vertices
with degree 2 and 3 in $G$.  Let $B$ be a plane graph with $n_B$ vertices and $e_B$
edges,  and with 3 vertices specified on its outer face. 
Form $G'$ by substituting $B$ into $G$.  Now $e(G') = (n_2+n_3)e_B$ and $n(G') =
n_2(n_B-1)+n_3(n_B-3/2)$.  Further, if $G$ has no cycle of length $\ell$ or
shorter, and $B$ has no cycle of length $\ell$, then $G'$ has no cycle of length $\ell$.
\end{lem}
\begin{proof}
Each vertex in $G$ gives rise to an edge-disjoint copy of $B$ in $G'$; thus
$e(G')=(n_2+n_3)e_B$.  Each vertex of degree 2 in $G$ contributes $n_B-1$ vertices to
$G'$, since exactly two of its vertices lie in two copies of $B$ in $G'$ (and
all others vertices lie in one copy of $B$).  Similarly, each vertex of degree 3 in $G$
contributes $n_B-3/2$ vertices to $G'$.  Finally, assume $G$ and $B$ satisfy the
hypotheses on the lengths of their cycles.  Now consider a cycle $C'$ in $G'$.
If $C'$ is contained entirely in one copy of $B$, then $C'$ has length not equal to $\ell$.  If $C'$ visits two or more copies of $B$, then $C'$ maps to a cycle
$C$ in $G$ with length no longer than the length of $C'$.  Since each cycle in $G$ has length longer than $\ell$, we are done.
\end{proof}

Now suppose that we plan to substitute some plane graph $B$ into a dense planar
graph of girth $\ell+1$.  Which $B$ should we choose?  Since $B$ must not
contain any $\ell$-cycle, a natural choice is a triangulation of order $\ell-1$.
Indeed, every such $B$ yields a graph that attains the bound in
Conjecture~\ref{main-conj}.  This is Corollary~\ref{obs1}, which follows from our next lemma.

\begin{lem}
\label{obs}
Let $G$ be a dense graph of girth $\ell+1$. Form $G'$ by substituting into $G$
a plane graph $B$ with 3 vertices specified on its outer face.  Now
$e(G')=\frac{e_B(\ell-1)}{(n_B-1)(\ell-1)-2} \left(n(G')-\frac{2(\ell+1)}{\ell-1}\right)$, where $e_B=e(B)$ and $n_B=n(B)$. 
\end{lem}

\begin{proof}
Let $G$ be a dense graph of girth $\ell+1$ on $n$ vertices, and let $n_2$ and
$n_3$ denote, respectively, its numbers of vertices with degree $2$ and $3$. 
Recall from Lemma~\ref{dense-lem} (with $g=\ell+1$) that  $n=(5\ell-5)\frac{k}{2}-\ell+3$
for some even integer $k$, that $n_3=10k-8$, and that $n_2=n-n_3$. 
Lemma~\ref{lem:construction} implies that $e(G')=(n_2+n_3)e_B=ne_B$ and that
$n(G') = n_2(n_B-1)+n_3(n_B-3/2) = (n-n_3)(n_B-1)+n_3(n_B-3/2)
=n(n_B-1)-n_3/2$. Now we show that $e(G')=\frac{e_B(\ell-1)}{(n_B-1)(\ell-1)-2}
(n(G')-\frac{2(\ell+1)}{\ell-1})$.  The final equality comes from substituting
for $n_3$ and simplifying (using that $n=(5\ell-5)\frac{k}2-\ell+3$).
\begin{align*}
\frac{e(G')}{n(G')-\frac{2(\ell+1)}{\ell-1}} &=
\frac{n e_B (\ell-1)}{(n(n_B-1)-n_3/2)(\ell-1)-2(\ell+1)}\\
&=\frac{e_B(\ell-1)}{(n_B-1)(\ell-1)-\frac{n_3(\ell-1)+4(\ell+1)}{2n}}\\
&=\frac{e_B(\ell-1)}{(n_B-1)(\ell-1)-2}.
\end{align*}
\aftermath
\end{proof}

\begin{cor}
\label{obs1}
The bound in Conjecture~\ref{main-conj} holds with equality for each graph
formed by substituting a triangulation on $\ell-1$ vertices into a dense graph
of girth $\ell+1$. 
\end{cor}
\begin{proof}
This follows from the above lemma when $B$ is a plane triangulation on $\ell-1$
vertices, so $n_B=\ell-1$ and $e_B=3(\ell-1)-6=3\ell-9$.  We get
\begin{align*}
\frac{e_B(\ell-1)}{(n_B-1)(\ell-1)-2}&= \frac{3(\ell-3)(\ell-1)}{(\ell-2)(\ell-1)-2}\\
&=\frac{3(\ell-3)(\ell-1)}{\ell^2-3\ell+2-2}\\
&=\frac{3(\ell-1)}{\ell}.
\end{align*}
\aftermath
\end{proof}

To beat the bound of Conjecture~\ref{main-conj}, it will suffice to instead
substitute into a dense graph of girth $\ell+1$ any triangulation with order
larger than $\ell-1$, as long as it has each cycle of length at
most $\ell-1$.  This is because the conjectured average degree is less than 6,
and is attained by substituting a triangulation of order $\ell-1$, as shown in
Corollary~\ref{obs1}.
However, the average degree of a triangulation tends to 6 (from below) as its
order grows.
For each
$\ell\in \{3,\ldots,10\}$, every triangulation on $\ell$ vertices is
Hamiltonian, i.e., it contains an $\ell$-cycle.  But for each $\ell\ge 11$,
there exists a triangulation on $\ell$ vertices with no $\ell$-cycle; this is a
consequence of Lemma~\ref{non-ham-lem}, which we prove next.
(In fact, much more is true, as we show in Section~\ref{denser-sec}.)

\begin{lem}
\label{non-ham-lem}
For every integer $t\ge 5$, there exist a plane triangulation with $3t-4$ vertices and
each cycle of length at most $2t$, and a plane triangulation with $3t-3$
vertices and each cycle of length at most $2t+1$.
\end{lem}
\begin{proof}
We start with a plane triangulation on $t$ vertices. First we add into the
interior of each face a new vertex, making it adjacent to each vertex on the
face.  Let $A$ denote the set of vertices in the original triangulation, and let
$B$ denote the set of added vertices.  Since $|A|=t$ and $|B|=2t-4$, the
resulting graph $G_1$ has order $3t-4$.  Further, $B$ is an independent set.  Thus,
on every cycle $C$, at least half of the vertices must be from $A$.  Hence, $C$ has
length at most $2|A|=2t$.

Now we obtain $G_2$ by adding a single vertex inside some face of $G_1$. It is easy to check that $G_2$ is a $(3t-3)$-vertex triangulation with each cycle of length at most $2t+1$.
\end{proof}

We have already outlined the proof of our main result.  We let $B$ be a plane
triangulation with no $\ell$-cycle, and with order at least $\ell$, as
guaranteed by Lemma~\ref{non-ham-lem}. We simply substitute $B$ into a dense graph
of girth $\ell+1$.  For completeness, we include more details in the proof of
Theorem~\ref{main1-thm}.

\begin{thm}
\label{main1-thm}
For each $\ell\ge 11$, Conjecture~\ref{main-conj} is false. 
In particular, whenever $k$ is positive if $\ell\ge 11$ and $\ell$ is
odd then,
$\exP(n,C_{\ell})\ge \frac{9(\ell-5)(\ell-1)}{(3\ell-13)(\ell-1)-4}
\left(n-\frac{2(\ell+1)}{\ell-1}\right)$ for
$n=((5\ell-5)\frac{k}{2}-\ell+3)(\frac{3(\ell-1)}{2}-5)-(5k-4)$ 
and if $\ell\ge 11$ and $\ell$ is even, then $\exP(n,C_{\ell})\geq \frac{3(3\ell-16)(\ell-1)}{(3\ell-14)(\ell-1)-4}
\left(n-\frac{2(\ell+1)}{\ell-1}\right)$ for
$n=((5\ell-5)\frac{k}{2}-\ell+3)(3(\frac{\ell}{2}-1)-4)-(5k-4)$. 
\end{thm}
\begin{proof}
Let $a_1:= \frac{9(\ell-5)(\ell-1)}{(3\ell-13)(\ell-1)-4}$ 
and $a_2:=\frac{3(3\ell-16)(\ell-1)}{(3\ell-14)(\ell-1)-4}$. 
Since $\ell \geq 11$,
easy algebra implies that $a_i>\frac{3(\ell-1)}{\ell}$, for each $i\in\{1,2\}$.
Thus, $a_i(n-\frac{2(\ell+1)}{\ell-1}) >
\frac{3(\ell-1)}{\ell}\left(n-\frac{2(\ell+1)}{\ell-1}\right) =
\frac{3(\ell-1)}{\ell}n-\frac{6(\ell+1)}{\ell-1}$ for each $i\in\{1,2\}$.  So
it suffices to show that $\exP(n,C_{\ell})\geq
a_1(n-\frac{2(\ell+1)}{\ell-1})$ when $\ell \ge 11$ and $\ell$ is odd; and that $\exP(n,C_{\ell})\geq
a_2(n-\frac{2(\ell+1)}{\ell-1})$ when $\ell \ge 11$ and $\ell$ is even (for the
claimed values of $n$).
Let $G$ be a dense graph of girth $\ell+1$. Recall that
$n(G)=(5\ell-5)\frac{k}{2} - \ell+3$ for some even integer $k$, and that $G$
has $10k-8$ vertices of degree $3$; let $n_3:=10k-8$.

When $\ell \ge 11$ and $\ell$ is odd, let $t_1:=\frac{\ell-1}{2}$ and
$n_{B_1}:=3t_1-4=\frac{3(\ell-1)}{2}-4$. We have $t_1\ge 5$; so by
Lemma~\ref{non-ham-lem}, there exists a plane triangulation $B_1$ with $n_{B_1}$
vertices and with each cycle of length at most $2t_1=\ell-1$. By Euler's
formula, $e_{B_1}=e(B_1)=3(3t_1-4)-6=9t_1-18=9(\frac{\ell-1}{2}-2)$. Form $G_1'$
by substituting $B_1$ into $G$. Lemma~\ref{lem:construction} implies that $G'$
is a plane graph with no cycle of length $\ell$, and that
$n(G_1')=n(G)(n_{B_1}-1)-n_3/2=((5\ell-5)\frac{k}{2}-\ell+3)(\frac{3(\ell-1)}{2}-5)-(5k-4)$.
By Lemma~\ref{obs}, we have 
\begin{align*}
    e(G_1') &=\frac{e_{B_1}(\ell-1)}{(n_{B_1}-1)(\ell-1)-2}\left(n(G_1')-\frac{2(\ell+1)}{\ell-1}\right)\\
    &=\frac{9(\ell-5)(\ell-1)}{(3\ell-13)(\ell-1)-4}
\left(n(G_1')-\frac{2(\ell+1)}{\ell-1}\right)\\
&=a_1\left(n(G_1')-\frac{2(\ell+1)}{\ell-1}\right).
\end{align*}
Hence, if $\ell \ge 11$ and $\ell$ is odd, then whenever $k$ is positive and
even and $n=((5\ell-5)\frac{k}{2}-\ell+3)(\frac{3(\ell-1)}{2}-5)-(5k-4)$,
we have $\exP(n, C_{\ell})\ge a_1\left(n-
\frac{2(\ell+1)}{\ell-1}\right)>\frac{3(\ell-1)}{\ell}n- \frac{6(\ell+1)}{\ell}.$

Now suppose $\ell \ge 11$ and $\ell$ is even. Let $t_2:=\frac{\ell}{2}-1$ and
$n_{B_2}:=3t_2-3=\frac{3\ell}{2}-6$. Form $G_2'$ by substituting $B_2$ into
$G$, where $B_2$ is a plane triangulation with $n_{B_2}$ vertices and each
cycle of $B_2$ has length at most $2t_2+1=\ell-1$. (The existence of $B_2$ is
guaranteed by Lemma~\ref{non-ham-lem}.) 
By Euler's formula, $e_{B_2}=e(B_2)=\frac{9\ell}{2}-24$.
Similarly, it follows from
 Lemma~\ref{lem:construction} that $G_2'$ is a plane graph with no cycle of
length $\ell$, and that $n(G_2')= n(G)(n_{B_2}-1)-n_3/2 =
((5\ell-5)\frac{k}{2}-\ell+3)(\frac{3\ell}{2}-7) - (5k-4)$. 
Lemma~\ref{obs} implies that
\begin{align*}
    e(G_2') &=\frac{e_{B_2}(\ell-1)}{(n_{B_2}-1)(\ell-1)-2}\left(n(G_2') 
-\frac{2(\ell+1)}{\ell-1}\right)\\
    &=\frac{3(3\ell-16)(\ell-1)}{(3\ell-14)(\ell-1)-4}
\left(n(G_2')-\frac{2(\ell+1)}{\ell-1}\right)\\
&=a_2\left(n(G_2')-\frac{2(\ell+1)}{\ell-1}\right)>\frac{3(\ell-1)}{\ell}n(G_2')- 
\frac{6(\ell+1)}{\ell}.
\end{align*}
This completes our proof.
\end{proof}

Now for each $\ell\ge 11$, we extend the construction in Theorem~\ref{main1-thm}
to all sufficiently large $n$ (which will prove the first sentence of
Theorem~\ref{over-thm}).  Our general idea is to build a counterexample
with order $n'$, larger than $n$, and delete vertices to get a counterexample of
order precisely $n$.  To see that this works, note that we can substitute
different gadgets for different vertices in a sparse planar graph of girth
$\ell+1$.  As long as each gadget has more than $\ell$ vertices, we will beat
the bound in Conjecture~\ref{main-conj}.  In fact, we still beat the bound if a
bounded number of gadgets have exactly $\ell$ vertices, and all other gadgets
have more vertices (this is only needed in the case that $\ell\in \{11,12\}$,
since that is when the gadget has precisely $\ell$ vertices).
So we follow the construction in
Theorem~\ref{main1-thm}, and then repeatedly remove vertices of degree 3 (that lie
in $B$ in Lemma~\ref{non-ham-lem}).  We can remove up to $t-4$ of these from
each gadget.  And the increase to the order of $G'$ when we increase $k$ in
Theorem~\ref{main1-thm} is less than $(5g-10)(3t-5)$.  So it suffices that the
number of vertices in the sparse planar graph $G$ is greater than
$\lceil(5g-10)(3t-5)/(t-4)\rceil\le 50(g-2)$.  
This proves the first sentence of
Theorem~\ref{over-thm}.

\section{Denser Constructions and a Revised Conjecture}
\label{second-sec}
\label{denser-sec}

In this short section, we construct counterexamples to Conjecture~\ref{main-conj}
that are asymptotically much denser than those in the previous section.  We
also propose two revised versions of Conjecture~\ref{main-conj}.

By iterating the idea in Lemma~\ref{non-ham-lem},
Moon and Moser~\cite{moon-moser} constructed planar triangulations where the
length of the longest cycle is sublinear in the order.  These triangulations
will serve as the gadgets in our denser constructions.

\begin{thm}[\cite{moon-moser}]
\label{MM-thm}
For each positive integer $k$ there exists a 3-connected plane triangulation
$G_k$ with $n(G_k)=\frac{3^{k+1}+5}2$ and with longest cycle of length less than
$\frac72n(G_k)^{\log_32}$.
\end{thm}

\begin{cor}
\label{MM-cor}
There exists a positive real $D_1$ such that for all integers $\ell\ge 6$ there
exists a plane triangulation $G_{\ell}$ with $n(G_{\ell})\ge
D_1\ell^{\lg_23}$ such that $G_{\ell}$ has no cycle of length at least $\ell$.
\end{cor}

Chen and Yu~\cite{chen-yu} showed that Theorem~\ref{MM-thm} is essentially best possible.

\begin{thm}[\cite{chen-yu}]
\label{CY-thm}
There exists a positive real $D_2$ such that every 3-connected $n$-vertex
planar graph contains a cycle of length at least $D_2n^{\log_32}$.
\end{thm}

We briefly sketch the Moon--Moser construction, which proves Theorem~\ref{MM-thm}.
For a more detailed analysis,
we recommend Section 2 of~\cite{chen-yu}.  Start with a planar drawing of $K_4$,
which we call $T_1$.  To form $T_{i+1}$ from $T_i$, add a new vertex $v_f$
inside each face $f$ (other than the outer face), making $v_f$ adjacent to each
of the three vertices on the boundary of $f$, see Figure~\ref{fig:Ti}.  It is each to check that the
order of $T_i$ is $3+ (1+3+\ldots+3^{i-1})\approx \frac{3^i}2$.

\newcommand{\deeprer}[4]{
\path($(#1)!0.5!(#2)$) -- node[pos=0.35](#4){} (#3);
\draw(#4)--(#3)(#1)--(#4)--(#2);
}

\newcommand{\deeprerB}[4]{
\path($(#1)!0.5!(#2)$) -- node[pos=0.55](#4){} (#3);
\draw(#4)--(#3)(#1)--(#4)--(#2);
}

\newcommand{\deeeprer}[4]{
\deeprer{#1}{#2}{#3}{#4}
\deeprerB{#1}{#2}{#4}{xx1}
\deeprerB{#1}{#3}{#4}{xx2}
\deeprerB{#2}{#3}{#4}{xx3}
}

\begin{figure}
\begin{center}
\begin{tikzpicture}[scale=3.5, thick]

\clip (-.5,-.19) rectangle (7,.9);

\draw
(0,0) node(a){}(60:1) node(b){}(0:1) node(c){}
(a)--(b)--(c)--(a);
\deeprer{a}{b}{c}{x}
\draw (0.5,-0.15) node[fill=white,draw=white]{\footnotesize{$T_1$}};

\begin{scope}[xshift=1.5cm]
\draw
(0,0) node(a){}(60:1) node(b){}(0:1) node(c){}
(a)--(b)--(c)--(a);
\deeeprer{a}{b}{c}{x}
\draw (0.5,-0.15) node[fill=white,draw=white]{\footnotesize{$T_2$}};
\end{scope}

\begin{scope}[xshift=3.0cm]
\draw
(0,0) node(a){}(60:1) node(b){}(0:1) node(c){}
(a)--(b)--(c)--(a);
\deeprer{a}{b}{c}{x}
\deeeprer{a}{b}{x}{y1}
\deeeprer{a}{c}{x}{y2}
\deeeprer{b}{c}{x}{y3}
\draw (0.5,-0.15) node[fill=white,draw=white]{\footnotesize{$T_3$}};
\end{scope}

\end{tikzpicture}
\end{center}
    \caption{Triangulations $T_1$, $T_2$, and $T_3$.}
    \label{fig:Ti}
\end{figure}
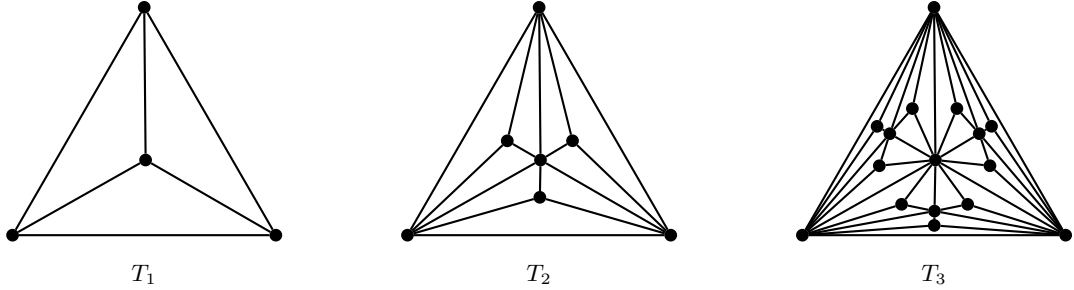

To bound the length of the longest cycle in $T_i$, we note that the
vertices added when forming $T_j$ from $T_{j-1}$ form an independent set, for
each $j$.  Thus, for any cycle in $T_i$, at most half of the vertices were added
at the final step.  Of those added earlier, at most half were added at the
penultimate step, etc.  So the length of a longest cycle grows roughly by a
factor of 2 at each step (while the order of $T_i$ grows roughly by a factor of
3).

To prove the second statement of Theorem~\ref{over-thm}, we substitute into a
sparse planar graph of girth $\ell+1$ a gadget with no cycle of length $\ell$,
as guaranteed by Corollary~\ref{MM-cor}.  We suspect this construction is extremal.
So we conclude with the following two conjectures, which are each best possible.

\begin{conj}
\label{conj1}
Fix $\ell\ge 7$, let $G$ be a dense graph of girth $\ell+1$, and let $B$ be a
$n$-vertex planar triangulation with no $\ell$-cycle, where $B$ is chosen to
maximize $n$.  If $G'$ is formed by substituting $B$ into $G$ and $n':=|V(G')|$,
then $\exP(n',C_l)=|E(G')|$.
\end{conj}

Proving Conjecture~\ref{conj1} seems plausible for some small values of $\ell$.
But proving it in general seems difficult.  So we also pose the following
weaker conjecture.  Note that Conjecture~\ref{conj2} would be immediately
implied by Conjecture~\ref{conj1} (together with Theorem~\ref{CY-thm}).

\begin{conj}
\label{conj2}
There exists a constant $D$ such that for all $\ell$ and for all sufficiently
large $n$ we have $\exP(n,C_{\ell})\le
\frac{3(D\ell^{\lg_23}-1)}{D\ell^{\lg_23}}n$.
\end{conj}

\section*{Acknowledgments}
Most research in this paper took place at the 2021 Graduate Research
Workshop in Combinatorics.  We heartily thank the organizers.  We also thank
Caroline Bang, Florian Pfender, and Alexandra Wesolek for early discussions on this problem.
\bibliographystyle{abbrvurl}

{\footnotesize{\bibliography{refs}}}
\end{document}